\documentclass[12pt]{article}
\usepackage[utf8]{inputenc}

\usepackage[utf8]{inputenc}
\usepackage{amsfonts}
\usepackage{amsthm}
\usepackage{amsmath}
\usepackage{amssymb}
\usepackage{stmaryrd}
\usepackage{yfonts}
\usepackage{MnSymbol}
\usepackage{mathtools}
\usepackage{fullpage}
\usepackage{color}
\usepackage{enumitem}
\usepackage{tikz}
\usepackage{capt-of}
\usepackage{tikzscale}

\usepackage{pb-diagram}
\usepackage[T1]{fontenc}
    \usepackage{lipsum}
    \usepackage{tocloft}
\usepackage{tikz-cd}
\date{\today}

\newtheorem{teo}{Theorem}[section]

\newtheorem{coro}[teo]{Corollary}
\newtheorem{lema}[teo]{Lemma}
\newtheorem{rem}[teo]{Remark}
\newtheorem{prop}[teo]{Proposition}
\newtheorem{defi}[teo]{Definition}
\newtheorem{ex}[teo]{Example}

\usepackage[all]{xy}

\DeclareMathOperator{\Sh}{\mathcal{S}}

\DeclareMathOperator{\N}{\mathbb{N}}

\newcommand{\cantor}{2^{\N}}

\newcommand{\dom}{\text{\sf dom}}
\newcommand{\im}{\text{\sf im}}
\newcommand{\su}{\subseteq}
\newcommand{\finitosg}{{\mathcal{P}_{\rm fin}(G)}}

\title{A topological correspondence between   partial actions of groups and  inverse semigroup actions}

\author{
	Luis Mart\'inez, H\'{e}ctor Pinedo and   Carlos Uzc\'ategui\\
	\small  Escuela de Matem\'{a}ticas\\
	\small Universidad Industrial de Santander\\
	\small Cra. 27 calle 9, Bucaramanga, Colombia\\
	\small  e-mail: luchomartinez9816@hotmail.com, hpinedot@uis.edu.co, cuzcatea@saber.uis.edu.co}

\begin{document}

\maketitle
\begin{abstract} 
We present some generalizations of the well known correspondence, found by  R. Exel, between partial actions of a group $G$ on  a set $X$ and  semigroup homomorphism of  $\Sh(G)$ on the semigroup $I(X)$ of partial bijections of $X,$ being $\Sh(G)$  an inverse  monoid introduced by Exel. We show that any unital premorphism  $\theta:G\to S$, where $S$ is an inverse monoid, can be extended to a semigroup homomorphism $\theta^*:T\to S$  for any    inverse semigroup  $T$ with $\Sh(G)\su T\su  P^*(G)\times G,$ being $P^*(G)$ the semigroup of non-empty subset of $G$, and such that  $E(S)$ satisfies some lattice theoretical condition. 
We also consider a topological version of this result. We present a minimal Hausdorff inverse semigroup topology on $\Gamma(X)$, the inverse semigroup of partial homeomorphism between open subsets of a locally compact Hausdorff space $X$. 
\end{abstract}

\noindent{
	}

\date

\section{Introduction}
 A   partial action of a group $G$ on a set $X$ is a unital premorphism $\theta:G\to I(X),$ where $I(X)$ is the inverse semigroup of all partial bijections between subsets of $X$ (more details are given in the next section). Partial actions of groups were introduced by R. Exel in \cite{E0} in the context of $C^*$-algebras and have widespread  in several branches of mathematics.
Exel \cite{E1} constructed  a monoid $\Sh(G)$ which plays a crucial role in the theory of partial actions and partial representations of groups  and $C^*$-algebras. He  found  a  one-to-one
 correspondence between partial actions of $G$ on  $X$ and  semigroup homomorphism of  $\Sh(G)$ on $I(X)$ (see  section \ref{BR}).    

A  reformulation of Exel's correspondence for inverse monoids says that there is a one-to-one correspondence between unital premorphisms  $\theta:G\to S$, where $S$ is an inverse monoid, and semigroup homomorphisms of $\theta^*:\Sh(G)\to S$.  
Kellendonk and Lawson  showed that $\Sh(G)$ is isomorphic to the Birget-Rhodes expansion  $\widetilde{G}^R$ of $G$ (see \cite{KL}). The semigroup  $\widetilde{G}^R$ can be seen as a subsemigroup of a semidirect product  $\mathcal{P}^*(G)\rtimes_\beta G$. 
We present  a generalization of  Exel's correspondence showing that such unital premorphisms $\theta$ can be extended to a semigroup homomorphism $\rho^*:T\to S$ where $T$ is an inverse semigroup with $\widetilde{G}^R\subseteq T\subseteq \mathcal{P}^*(G)\rtimes_\beta G$  (see Theorem \ref{Text}). To obtain such extension some lattice theoretical conditions are imposed on $E(S)$. Other algebraic extensions of  Exel's correspondence  have being obtained by replacing, for instance,   the group $G$ by an inductive groupoid \cite{BFP}, an inverse semigroup \cite{BE},  a   weakly left $E$-ample semigroup  \cite{GH1} or  an inductive constellation  \cite{GH}.  The main idea in those works  was to find an appropriate version of the Birget-Rhodes expansion to get such extensions. 

We also aim  to explore  a topological version of  Exel's correspondence.  Suppose $S$ is topological inverse semigroup, $G$ a topological group and $\theta:G\to S$ a continuous unital premorphism. We show that the extension $\theta^*: \widetilde{G}^R\to S$ is also continuous assuming, as above, some lattice theoretical requirements on $E(S)$, namely, it must have a basis consisting of subsemilattices (see Theorem \ref{ExelCorres}).

A particular important inverse semigroup is $\Gamma(X)$, the collection of all homeomorphisms between open subsets of a topological space $X$. A unital premorphism $\theta:G\to \Gamma(X)$ is called a topological partial action  (see \cite{abadie,KL}).  In order to treat topologically the problem of extending $\theta$, we first show that an appropriated version of the classical compact-open topology on $\Gamma(X)$ is suitable for our purposes.  The natural environment  is to work with a locally compact Hausdorff space $X,$ which is also the  realm where  partial actions of groups on $C^*$-algebras  are studied \cite{E2}. There are some works on topologies on spaces of  partial functions with closed domain   (see \cite{DiConcilioetal2000} and references therein); but, apparently there are less works  when  the domain is open,  the earliest we have found is the work of Abd-Allah and Brown
 \cite{AbdBrown} who showed  that the compact-open topology makes $\Gamma(X)$ a topological inverse semigroup (this approach  has been used also in differential geometry \cite{AM}). However, the compact-open topology on $\Gamma(X)$  is not $T_1$. We define  a Hausdorff extension of it which turns out  to be  minimal in a natural class  of inverse semigroup topologies on $\Gamma(X)$. Moreover, it also has  the property that the collection of idempotents $E(\Gamma(X))$ is compact (see  Theorem \ref{hco-minimal} and Proposition \ref{hco-compact}).   For completing  the topological setting,  we need to endow $\widetilde{G}^R$ with a topology. This problem  was already explored by Choi \cite{Choi2013}  who considered  an inverse semigroup $\widetilde{G}_c^R\supseteq \widetilde{G}^R$  endowed with the subspace topology as a subset of $K(G)\times G$, where $K(G)$ is the hyperspace of compact subsets of $G$ with  the  Vietoris topology.  After all this preparation, we  finally show that there is  semigroup homomorphim $\theta^*:\widetilde{G}_c^R\to \Gamma(X)$ which is a continuous extension of $\theta$ (see Theorem \ref{ChoisExtension}).
  
The paper is organized as follows. Section 2 contains some preliminary facts and notations. In section 3 we define a Hausdorff extension of the compact-open topology on $\Gamma(X)$,  we show it makes it a topological inverse semigroup and works well with  partial actions of topological groups.     We prove  that the collection of idempotents $E(\Gamma(X))$ is homeomorphic to the space of closed subsets of $X$ with the Fell topology. We also show an extension result  that will be needed later for studying the Birget-Rhodes expansion of a topological group (see Proposition \ref{ext-to-hyp}). Finally, in section 4 we show the main results.   In order to obtain a  purely algebraic result about a generalization of Exel's correspondence (see Theorem \ref{Text}), we need to  introduce in  Definition \ref{meetc} a lattice-theoretical notion  on inverse semigroups which allows us to obtain in  Lemma \ref{lemainf}  a semigroup homomorphism which plays a crucial role in our generalization, after that we show a topological version of this correspondence (Theorem \ref{ExelCorres}).  We also extend this result  showing that any continuous unital premorphism from $G$ to a closed subsemigroup of $\Gamma(X)$ can be extended to a  continuous morphism from  $\widetilde{G}_c^R$ (Theorem \ref{ChoisExtension}). We observe in Example \ref{fex}  that, in general,  this extension is not unique.

\section{Preliminaries}
A semigroup $S$ is called {\it regular,}  if for any $s \in S$ there is $s^{-1} \in S$  such that $ss^{-1}s = s$ and $s^{-1}ss^{-1} = s^{-1}.$ A regular semigroup $S$ is said to be {\it inverse,} if each $s \in S$ has a unique $s^{-1}.$ Inverse semigroups are precisely those regular semigroups whose idempotents commute \cite[Theorem 3]{Law}. Each inverse semigroup admits the {\it natural partial order} defined by $s \leq t \Leftrightarrow s = et $ for some idempotent $e\in S.$  The set of idempotents of $S,$ for which we use the standard notation $E(S),$ is a meet semilattice under $\leq$,  where the meet $e \wedge f=ef$ for  $e, f \in E(S).$ The general reference for inverse semigroups is \cite{Law}.

Let $X$ be a set, denote by $I(X)$ the inverse monoid of partial bijections of $X$ (that is bijections between subsets of $X$) the multiplication rule on $I(X)$ is given
by composition of partial maps in the largest domain where it makes sense, that is if $f,g\in I(X)$ then $\dom(fg)=g^{-1} (\im g\cap \dom f)$ and  $\im(fg)=f(\im g\cap \dom f)$. An important classical result on inverse semigroup theory is  the Vagner-Preston representation Theorem which says that every inverse semigroup is  isomorphic to a subsemigroup of $I(X)$  (see, for instance,  \cite[Theorem 1, pag. 36]{Law}).   

According to  \cite{KL},  a function  $\theta: T\to S$ between inverse semigroups is a  {\it premorphism}  if for all $s,t\in T,$ $\theta(t^{-1})=\theta(t)^{-1}$ and $\theta(s)\theta(t)\leq  \theta(st).$
If  $S$ and $T$ are both monoids, then a  premorphism $\theta$ is called {\it unital} if $\theta(1)=1.$ When $T=G$ is a group,  it follows by    \cite[Proposition 2.1]{KL}  that  $\theta:G\to S$  is a unital premorphism, if and only if,
\begin{itemize} 
\item $\theta(1)=1.$
\item $\theta(g^{-1})\theta(g)\theta(h)=\theta(g^{-1})\theta(gh),$ for any $g,h\in G.$
\end{itemize}
It is not difficult to see that $\theta$ also satisfies   $\theta(g)\theta(h)\theta(h^{-1})=\theta(gh)\theta(h^{-1}) $ for $g,h\in G.$ A  {\em (set theoretic) partial action} of $G$ on $X$ is a unital premorphism $\theta:G\to I(X).$  Equivalently,  a partial action of  $G$ on  $X$ is  a family of partial bijections  $\alpha =\{\alpha_g : X_{g^{-1}} \to  X_g : g \in  G\}$ such that the corresponding map $\alpha : G\ni g \mapsto \alpha_g\in I(X)$ is a unital premorphism. For a detailed account about recent developments on partial actions the interested reader may consult  \cite{DO, E2}.

Let $X$ be a locally compact Hausdorff  space. $CL(X)$ denotes the collection of all closed subsets of $X$ (including the emptyset).   Consider the following subsets of $CL(X)$ for $V\su X$ open:
\[
V^+=\{A\in CL(X):\; A\su V\}
\]
and
\[
V^-=\{A\in CL(X):\; A\cap V\neq\emptyset \}.
\]
The {\em Fell} topology on $CL(X)$ has as a subbasis  the sets $V^-$ for $V$ open and $W^+$ for $W$ open with compact complement.
It is known that $CL(X)$ with the Fell topology is a compact Hausdorff space (see, for instance, \cite[Theorem 5.3.1]{Beer1993}).
The {\em Vietoris} topology on $CL(X)$ is generated by $V^+$ and $V^-$  with $V\su X$ open. The hyperspace  $K(X)$  is the collection of all compact subsets of $X$.  The Vietoris topology on $K(X)$  has the  following sets  as a basis:
 \[
 N(U_1,\cdots,U_n)=\{K\in K(X):\; K\su U_1\cup\cdots\cup U_n, \,K\cap U_i\neq \emptyset\; \text{for all $i\leq n$}\}
 \]
where each $U_i$  is open in $X$. Notice that  $K(X)$ is a topological semigroup under the map $(A,B)\mapsto A\cup B$.

\section{The semigroup of continuous maps with open domain}\label{cop}

The purpose of this section is  to present some natural topologies  that  will be needed to study continuous actions of topological inverse semigroups.  %The classical Wagner-Preston theorem says that any inverse semigroup is isomorphic to a  subsemigroup of $I(X)$.  
When $X$ is a topological space, a particularly important example is the semigroup of continuous partial homeomorphism with open domain. The main objective of this section is to show that this type of semigroups are topological under an appropriate version of the compact-open topology.  We start by fixing some notation. 

Let $X$ and $Y$ be Hausdorff spaces.  The space of partial continuous functions is the following:
\[
 C_{od}(X,Y)=\{f:D\to Y: \; \text{$D\su X$ is open and $f$ is continuous}\}.
 \]
When $X=Y$ we write $C_{od}(X)$ instead of  $C_{od}(X,X)$. 
Notice that $C_{od}(X)$ is a semigroup under the usual composition of partial functions.

 \subsection{The compact open topology}
 
 Now we introduce the compact-open topology on $C_{od}(X,Y)$, which is the most natural topology on spaces of continuous functions. As we said in the introduction, the oldest reference we have found about this topology on spaces of partial functions is \cite{AbdBrown}. 
For $K\su X$ and $V\su Y$, let 
 \[
 \langle K;V\rangle=\{f\in C_{od}(X,Y):\; K\su \dom(f)\; \& \; f(K)\su V\}.
 \]
 Let $\tau_{co}$ be the topology generated by  the sets $\langle K;V\rangle$, with $K$ compact  and $V$ open, as a subbasis.  
  
 \begin{teo}
 \label{prop-basicas}
 Let $X$ be a locally compact  Hausdorff space and $Y$ be a Hausdorff space. Then 
 \begin{itemize}
\item[(i)] The topology $\tau_{co}$ is $T_0$ and not necessarily $T_1$.
 
 \item[(ii)] Let 
 \[
C_{od}(X,Y)\ast X=\{(f,x)\in C_{od}(X,Y)\times X:\; x\in \dom(f)\}
 \]
and  $ev: C_{od}(X,Y)\ast X\to Y$ be the evaluation map given by $ev(f,x)=f(x)$. Then  $C_{od}(X,Y)\ast X$ is an  open subset of $C_{od}(X,Y)\times X$ and $ev$ is continuous with respect to the compact-open topology.

\item[(iii)] Let $X,Y,Z$ be Hausdorff spaces with $X$ and $Y$ locally compact. Then $c:C_{od}(X,Y)\times C_{od}(Y,Z)\to C_{od}(X,Z)$ given by $c(f,g)=g\circ f$ is continuous with respect to  the compact-open topology.  Thus $C_{od}(X)$ is a topological semigroup with the compact-open topo\-logy.

 \end{itemize}
 
 \end{teo}
 
\proof

$(i)$ To see that the compact-open topology is $T_0$, let $f,g\in C_{od}(X,Y)$ with $f\neq g$. There are some cases to be considered. Suppose $ \dom(f)= \dom(g)$. Then there is  $x\in \dom(f)$ such that  $f(x)\neq g(x)$. Thus $g\in \langle \{x\};Y\setminus\{f(x)\}\rangle$ and 
$f\nin  \langle \{x\};Y\setminus\{f(x)\}\rangle$. Suppose there is $x\in \dom(f)\setminus \dom(g)$. Then  $f\in\langle \{x\}; Y\rangle$ and $g\nin \langle \{x\};Y\rangle$. This topology is not $T_1$, in fact, suppose  $g\su f$, $g\neq f$ and $g\in \langle K; V\rangle$.  Then $f\in \langle K;V\rangle$. 
 
 $(ii)$ and $(iii)$ are   \cite[Proposition 7]{AbdBrown} and   \cite[Proposition 9]{AbdBrown}, respectively.  \endproof%To see that $C_{od}(X,Y)\ast X$ is open, let $(f, x)\in C_{od}(X,Y)\ast X$. Let $V$ open such that $x\in V\su \overline{V}\su \dom(f)$ and $\overline{V}$ is compact.  Then $(f,x)\in \langle\overline{V};Y\rangle\times V\su C_{od}(X,Y)\ast X$. 
 
% 
%To see that $ev$ is continuous,  let $V\su Y$ open and fix $(f_0,x_0)\in C_{od}(X,Y)\ast X$ such that  $f_0(x_0)\in V$. Let $W\su \dom(f_0)$ be open such that 
%$x_0\in W\su\overline{W}\su f_0^{-1}(V)$ and $\overline{W}$ compact. Then 
%\[
%(g,x)\in \langle\overline{W}; V\rangle\times  W\Rightarrow ev(g,x)=g(x)\in V.
%\]
%
% $(iii)$ Let $K\subseteq X$ be compact and $V\subseteq Z$ be open. Let $(f,g)$ be such that $g\circ f\in \langle K;V\rangle$, i.e. $g(f(K))\subseteq V$. Thus $f(K)\su g^{-1}(V)$. Pick $W$ open such that $f(K)\su W\su\overline{W}\su g^{-1}(V)$. Hence $(f,g)\in \langle K;W\rangle \times \langle \overline{W}; V\rangle $ and this open set works.

\begin{teo}
\label{minimality2} Let $X$ be a locally compact  Hausdorff space, $M$ be  a topological space and $h : M \to C_{od}(X)$ be  a map.   Let 
\[
M\ast X=\{(m,x)\in M\times X:\; x\in \dom(h(m))\}=\{(m,x)\in M\times X:\; (h(m), x)\in C_{od}(X)\ast X\}.
\]
The following are equivalent:

\begin{itemize}
\item[(i)]  $h$ is continuous with respect to the compact-open topology.

\item[(ii)] $M\ast X$ is open in $M\times X$ and the map $a:M\ast X\to X$ given by $a(m, x)=h(m)(x)$ is continuous.
\end{itemize}
In particular, the compact-open topology is the coarsest topology that turns $C_{od}(X)$ into a topological semigroup in such a way that the evaluation map $ev:C_{od}(X) \ast X\to X$ is continuous and $C_{od}(X) \ast X$ is open in $C_{od}(X) \times X$.

\end{teo}

\proof $(i) \Rightarrow (ii)$ Suppose $h$ is continuous with respect to the compact-open topology. Since $a(m,x)=ev(h(m),x)$,  by Proposition \ref{prop-basicas},  we get that $a$ is continuous.  As $C_{od}(X)\ast X$ is open and the function $M\times X\ni (m,x)\mapsto (h(m), x) \in C_{od}(X)\times X$ is continuous, then   $M\ast X$ is open. 

 $(ii) \Rightarrow (i)$ Suppose now that $a$ is continuous. Let $K\su X$ be a compact set and $V\su X$ be an open set. 
Let $O=h^{-1}(\langle K;V\rangle)$. Then 
\[
O=\{m\in M: \; h(m)(K)\su V\}=\{m\in M:\; K\su h(m)^{-1}(V)\}=\{m\in M:\; \{m\}\times K\su a^{-1}(V)\}.
\]
To see that  $O$ is open, let $m\in O$. Since  $M\ast X$ is open in $M\times X$,  $a^{-1}(V)$ is open  in $M\ast X$ and $K$ is compact, there are $W\su M$, $U\su X$ open sets such that $m\in W$, $K\su U$ and $W\times U\su a^{-1}(V)$.  Then $m\in W\su O$. 

For the last claim, let $\tau$ be a semigroup topology on $C_{od}(X)$ and $M=(C_{od}(X),\tau) $. Suppose  the evaluation map $a:M\ast  X\to X$  is continuous and $M\ast X$ is open in $M\times X$.   Let $h:M\to (C_{od}(X), \tau_{co})$ given by $h(f)=f$. Since  (ii) implies (i), $h$ is continuous, that is, $\tau_{co}\su \tau$.
\endproof

\subsection{The inverse semigroup of partial homeomorphisms}
  
We adopt the notation from \cite{AbdBrown}  and  set
$$
 \Gamma(X)=\{f|\;\text{$f:\dom(f)\to \im(f)$ is an homeomorphism with $\dom(f), \im(f)$  open subsets of $X$}\}
.$$
 By  an abuse of notation, we will regard $\Gamma(X)$ as a submonoid of $C_{od}(X)$. The elements of $ \Gamma(X)$ are the so-called {\it partial homeomorphisms }of $X.$
   
\begin{rem}{\rm 
The set $\Gamma(X)$ is an inverse semigroup with the usual composition of partial functions.   An inverse subsemigroup  $S$ of $\Gamma(X)$ for which $(f\restriction V)\in S,$ for each $f\in S$ and $V\subseteq \dom(f)$ open in $X$ is called a {\it pseudogroup of local transformations of X}. Pseudogroups are used in many areas of topology and differential geometry (see for instance \cite{AbdBrown, AM}  and the references therein).}
 
 \end{rem}

It is well known that the group $\mathcal{H}(X)$ of homeomorphisms of a compact Hausdorff space $X$ with the compact-open topology is  a topological group.  However, this is no longer the case when $X$ is not compact, the problem is that the inversion map $f\to f^{-1}$ might be discontinuous (some examples are presented in \cite{dijkstra,dijkstra2010}). Arens \cite{arens46} showed that if $X$ is locally compact and locally connected, then $\mathcal{H}(X)$ is a topological group with the compact-open topology (see also \cite{dijkstra}). We see below that a similar problem occurs with $\Gamma(X)$ but even when $X$ is compact. 

From this point on,  to avoid writing $\langle K; V\rangle \cap \Gamma(X)$, we assume that all   $\tau_{co}$ open sets are relativized to $\Gamma(X)$. 

\begin{ex}
{\em  We show that $\Gamma(\cantor)$ is not a topological inverse semigroup with the compact-open topology $\tau_{co}$. Fix an element  $x\in \cantor$ in  the Cantor space. Let $X=\cantor\setminus \{x\}$. Dijkstra \cite{dijkstra} (see also \cite{dijkstra2010}) have shown that the inversion map $f\to f^{-1}$ in $\mathcal{H}(X)$ is not continuous with respect to the compact-open topology. Since $X$ is open in $\cantor$,  $\mathcal{H}(X)\su \Gamma(\cantor)$ and thus the inversion map is not continuous as a map on $\Gamma(\cantor)$. 
}
\end{ex}

We need to enlarge the compact-open topology to make continuous the inversion map. 
The easiest way is as follows. Let $\tau_{ico}$ be the topology on $\Gamma(X)$ generated by the following collection of sets, where $K$ is compact and $V$ is open:
$
\langle K;V\rangle$ and 
$
(\langle K;V\rangle)^{-1}.
$
Notice that
\[
(\langle K;V\rangle)^{-1}=
\{f\in \Gamma(X):\; f^{-1}\in\langle K;V\rangle\}=\{f\in \Gamma(X):\; K\su f(V)\}.
\]

\begin{teo}
\label{topico} Let $X$ be a locally compact Hausdorff space. Then $(\Gamma(X), \tau_{ico})$ is a topological inverse semigroup. 
\qed
\end{teo}

\proof Let $c:\Gamma(X)\times \Gamma(X)\to \Gamma(X)$ be  $c(f,g)=f\circ g$. An argument, analogous to that in the proof of Theorem \ref{prop-basicas}, shows that $c^{-1}(\langle K;V\rangle)$ and $c^{-1}(\langle K;V\rangle^{-1})$ are $\tau_{ico}$-open, for every $K$ compact and $V$ open. Finally, it is obvious  from the definition of $\tau_{ico}$, that the inversion map $f\mapsto f^{-1}$   is continuous with respect to the topology $\tau_{ico}$.
\endproof

Similar to what happens with $\mathcal{H}(X)$, under the hypothesis of local connectedness, $\tau_{co}$ is an inverse semigroup topology. We include just a sketch of the proof since the argument is entirely similar to that given in \cite{arens46, dijkstra}.

\begin{teo}
Let $X$ be a locally compact and locally connected Hausdorff space. Then $\tau_{co}$ is an inverse semigroup topology on $\Gamma(X)$.
\end{teo}

\proof It suffices to show that $\tau_{ico}\su\tau_{co}$, equivalently,  $\langle K; V\rangle^{-1}=\{f\in \Gamma(X):\; K\su f(V)\}$  belongs to $\tau_{co}$ for every $K$ compact and $V$ open in $X$. Let $f\in \Gamma(X)$ be such that $K\su f(V)$. Let $A\su f^{-1}(K)$ be a finite set  such that $f^{-1}(K)\su \bigcup_{x\in A} U_x$ where  $U_x$ is an open connected set such that  $x\in U_x$ and $C_x=\overline{U_x}\su\dom(f)$ is compact.  Let $C=\bigcup_{x\in A} C_x$. By compactness, let $C'$ be a compact set such that $C\su int (C')\su C'\su \dom(f)$. Let $\partial C'$ be the boundary of $C'$. Consider the following $\tau_{co}$ open set
\[
O=\langle (X\setminus V)\cap M; X\setminus K\rangle\cap \langle \partial C'; X\setminus f(L)\rangle\cap \bigcap_{x\in A}\langle \{x\}; U_x\rangle.
\]
We left to the reader the verification that $f\in O\su \langle K; V\rangle^{-1}$ (see \cite{dijkstra}).
\endproof

To end this section, we  show the interplay between the compact-open topology and partial actions of topological groups.    
A {\em topological partial action} of $G$ on $X$ is a unital premorphism $\theta: G\to \Gamma(X).$ 
Part (ii) of the following theorem provides the natural notion of  partial action (resp. nice partial action), as it was introduced in  \cite[Definition 1.1]{abadie}  (resp. \cite[P. 105]{KL}).

\begin{teo}
\label{minimality3} Let $X$ be a locally compact  Hausdorff space, $G$ be  a topological group and  $\theta : G \to \Gamma(X)$   be a  topological  partial action.  Let
\[
G\ast X=\{(g,x)\in G\times X:\; x\in \dom(\theta(g))\}=\{(g,x)\in G\times X:\; (\theta(g), x)\in C_{od}(X)\ast X\}.
\]
The following are equivalent:
\begin{itemize}
\item[(i)]  $\theta$ is continuous with respect to the  topology $\tau_{ico}$.

\item[(ii)] $G\ast X$ is open in $G\times X$ and $m:G\ast X\to X$ given by $m(g, x)=g\cdot x$ is continuous.
\end{itemize}

\end{teo}

\proof  We check that the proof of Theorem \ref{minimality2} works in this case too. To see  that $(i)$ implies $(ii)$, we just observe that $\Gamma(X)\ast X=(C_{od}(X)\ast X)\cap (\Gamma(X)\times X)$, so $\Gamma(X)\ast X$ is open in $\Gamma(X)\times X$. The rest of the argument is the same. 

For $(ii)$ implies $(i)$, we only need to verify that $\theta^{-1}(\langle K, V\rangle^{-1})$ is open in $G$. Since  $\theta(g^{-1})=\theta (g)^{-1}$ for all $g\in G$,  we have
\[
\theta^{-1}(\langle K; V\rangle^{-1})=(\theta^{-1}(\langle K; V\rangle))^{-1}.
\]
By  Theorem \ref{minimality2},  $\theta^{-1}(\langle K; V\rangle)$ is open in $G$. Since 
 $g\mapsto g^{-1}$ is a homeomorphism, the set above is open in $G$. 
\endproof

\subsection{An inverse semigroup Hausdorff topology on $\Gamma(X)$}

As a consequence of the proof of  Theorem \ref{prop-basicas}(i), it follows that   $\tau_{ico}$ is not always $T_1$. In this section we extend  $\tau_{ico}$ to obtain a Hausdorff topology on $\Gamma(X)$.  We start presenting a general result about inverse semigroup topologies.

\begin{prop}
\label{inverset2}  Let $S$ be an inverse semigroup and $\tau$ be an inverse semigroup topology on $S$. Then $\tau$ is $T_2$ iff the natural partial order $\preceq$ is closed in $S\times S.$ 
In particular, an inverse semigroup topology $\tau$ on $\Gamma(X)$ is $T_2$ iff the relation $f\subseteq g$ is closed. 

\end{prop}
\proof Recall that $\tau$ is $T_2$ iff $\Delta=\{(s,s): \; s\in S\}$ is closed. Let $\varphi:S\times S\to S\times S$ given by $\varphi(s,t)=(s,ss^{-1}t)$. Notice that $\varphi$ is continuous and $s\preceq t$ iff $\varphi(s,t)\in \Delta$.  Thus if $\Delta$ is closed, then $\preceq$ is closed. Conversely, suppose $\preceq$ is closed. Notice that $(s,t)\in \Delta$ iff $s\preceq t$ and $t\preceq s$. Thus $\Delta$ is closed. 
\endproof

To define our new topology, consider the functions ${\sf D}, {\sf I}:\Gamma(X)\to CL(X)$ given by ${\sf D}(f)=X\setminus \dom(f)$ and ${\sf I}(f)=X\setminus \im(f)$.  We show that any Hausdorff inverse semigroup topology on $\Gamma(X)$ will necessarily make  $\sf D$ and $\sf I$ continuous for the Fell topology.

\begin{prop}
Let $\tau$ be an inverse semigroup Hausdorff topology on $\Gamma(X)$ extending $\tau_{ico}$.  Then $\sf D$ and $\sf I$ are continuous for the Fell topology.
\end{prop}

\proof First of all, observe that  ${\sf I}(f)={\sf D}(f^{-1})$, thus if ${\sf D}$ is continuous, so is ${\sf I}$.
Let $V,W$ be  open sets of  $X$ with $W^c$ compact. We recall that the subbasic open sets of $CL(X)$ are 
$V^-=\{K\in K(X): K\cap V\neq \emptyset\}$ and $W^+ = \{K\in K(X): K\su W\}$. Notice that ${\sf D}^{-1}(W^+)= \langle W^c, X\rangle$  which belongs to $\tau_{co}$. On the other hand,   
\[
{\sf D}^{-1}(V^-)= \{f\in\Gamma(X):\; {\rm id}_V\not\su f^{-1}\circ f\}.
\] 
Thus, we only need to show that ${\sf D}^{-1}(V^-)$ is $\tau$-open. This follows using the fact that $\su$ is the natural order in $\Gamma(X)$ and  Proposition \ref{inverset2} which says that $\su$ is a $\tau$-closed subset of $\Gamma(X)\times \Gamma(X)$ (observe that the function $f\mapsto f^{-1}\circ f$ is $\tau$-continuous).
\endproof

In view of the previous result, we introduce an extension of $\tau_{ico}$ as follows. Let $\tau_{hco}$ be the topology generated by $\tau_{ico}$ together with the sets ${\sf D}^{-1}(U)$ and ${\sf I}^{-1}(U)$ for $U$ a Fell open subset of $CL(X)$.

\begin{prop}
Let $X$ be a locally compact Hausdorff space. Then $(\Gamma(X), \tau_{hco})$ is a Hausdorff topological inverse semigroup.
\end{prop}
	
\proof 
\noindent To see that $\tau_{hco}$ is Hausdorff, let $f,g\in \Gamma(X)$ with $f\neq g$. There are two cases to be considered. (a) Suppose $\dom(f)=\dom(g)$. Then there is $x\in \dom(f)$ such that $f(x)\neq g(x)$. Let  $U_1$ and $U_2$ be  disjoint open sets in $X$ with $f(x)\in U_1$ and $g(x)\in U_2$. Then $\langle \{x\}, U_1\rangle$ and $\langle \{x\}, U_2\rangle$ separates $f$ from $g$.  (b) Suppose $\dom(f)\neq\dom(g)$. Since $CL(X)$ is Hausdorff, there are disjoint open sets $O$ and $U$ in $CL(X)$ such that $\dom(f)^c\in O$ and $\dom(g)^c\in U$. Then ${\sf D}^{-1}(O)$ and ${\sf D}^{-1}(U)$ are the required disjoint $\tau_{hco}$-open sets separating $f$ from $g$. 

To see that the inversion map $i$ is continuous just observe that $i^{-1}({\sf D}^{-1}(V^+))={\sf I}^{-1}(V^+)$ and the same with $V^-$. 

Let $c(f,g)=f\circ g$. We check that $c$ is continuous. Let $V\su X$ be an open set.  First of all, $c^{-1}({\sf D}^{-1}(V^+))=c^{-1}(\langle V^c, X\rangle)$ which belongs to $\tau_{ico}$ when $V^c$ is compact. On the other hand, we claim that  $c^{-1}({\sf D}^{-1}(V^-))$ is equal to the following set:
\[
\Gamma(X)\times {\sf D}^{-1}(V^-)\;\cup\;\bigcup\;\{{\sf D}^{-1}(W^-)\times \langle \overline{W}, V\rangle^{-1}:\; \emptyset\neq W\su X\;\text{is  open with $\overline{W}$ compact}\}.
\]
In fact, let $f, g$ be such that  $f\circ g\in {\sf D}^{-1}(V^-)$. There are two cases to be considered. If $g\in 
{\sf D}^{-1}(V^-)$, there is nothing to show. Otherwise, suppose  that  $\dom(g)^c\cap V=\emptyset$, i.e. $V\su \dom(g)$. Thus, as 
$\dom(f\circ g)\in V^-$, there is $x\in V$ such that $g(x)\not\in \dom(f)$. Let $W\su X$ be an open  set such that $g(x)\in W\su\overline{W}\su g(V)$ and $\overline{W}$ is compact. Hence $f\in {\sf D}^{-1}(W^-) $ and $g\in\langle \overline{W}, V\rangle^{-1}$.

Conversely, if $(f,g)\in\Gamma(X)\times {\sf D}^{-1}(V^-)$, clearly $f\circ g\in {\sf D}^{-1}(V^-)$. Now suppose that $(f,g)\in {\sf D}^{-1}(W^-)\times \langle \overline{W}, V\rangle^{-1}$ for some non empty open subset $W\su X$. Let $y \in W\cap \dom(f)^c$. Since $\overline{W}\su g(V)$, there is $x\in V$ such that $y=g(x)$. Therefore $x\not\in \dom(f\circ g)$ and $(f\circ g)\in {\sf D}^{-1}(V^-)$.

Finally, observe that  ${\sf I}^{-1}(V^+)= i^{-1}({\sf D}^{-1}(V^+))$ which shows that  $c^{-1}({\sf I}^{-1}(V^+))$ is $\tau_{hco}$-open. Analogously for $V^-$. 
\endproof

\bigskip

From the previous result and Theorem \ref{minimality2} we get the following. 

\begin{teo}
\label{hco-minimal}
Let $X$ be a locally  compact Hausdorff space. Then $\tau_{hco}$ is the smallest Hausdorff topology on $\Gamma(X)$ that makes it a topological inverse semigroup such that  the evaluation map $ev$ is continuous and $\Gamma(X)\ast X$ is open in $\Gamma(X)\times X$.  
\end{teo}

\begin{rem}
\label{taupp} {\rm
When $X$ is a discrete space, $\Gamma(X)$ is the inverse symmetric semigroup $I(X)$ of all partial bijections between subsets of $X$. In this case,  it is easy to see that  $\tau_{hco}$ is generated by the sets $v(x,y)=\{f\in I(X):\; f(x)=y\}$, $w_1(x)=\{f\in I(X):\; x\nin \dom(f)\}$ and $w_2(x)=\{f\in I(X): \; x\nin\im(f)\}$. This topology  was studied in \cite{elliott2020,PerezUzca2020}.}
 \end{rem}

 A natural example in the various
areas where partial actions of groups appear are  the  {\it induced partial actions,} whose construction in the topological context  is as follows:  Let  $X$ be a topological space,  $G$ be a topological group, $a:G\times X\to X$ be an action and  $Y\subseteq X$  be an open set. For $g\in G$, set $Y_g=\{y\in Y\mid a(g,x)=y \,\,\text{for some}\,\, x\in Y\}$ and  $\tilde{a}:G\to  \Gamma(Y),$ where $\tilde{a}_g:Y_{g^{-1}}\ni y\mapsto a(g,y)\in  Y_g$. Then $\tilde{a}$ is a unital premorphism called the  induced partial action of $a$ to $Y.$ Our next result strengthens Theorem \ref{minimality3}, we show that a continuous action induces a continuous partial action with respect to $\tau_{hco}.$

\begin{prop}\label{inducon}
Let $X$ be locally compact Hausdorff  space  and $G$ be a topological group. Let  $a:G\times X\to X$ be a continuous action and  $Y\subseteq X$ be an open set. Then, the induced partial action  $\tilde{a}:G\to  (\Gamma(Y), \tau_{hco})$ is continuous.
\end{prop}

\proof
We first observe that 
\[
G\ast Y=\{(g,y)\in G\times Y:\; y\in Y_{g^{-1}}\}=\{(g,y)\in G\times Y:\; a(g^{-1}, y)\in Y\}
\]
is open in $G\times Y$ as $a$ is continuous and $Y$ is open in $X$. By  Theorem \ref{minimality3},  $\tilde{a}$ is open with respect to $\tau_{ico}$. Thus, it suffices to show that $O=\tilde{a}^{-1}({\sf D}^{-1}(V^-))$ is open for every non empty open set $V\subseteq Y$.
Let $g_0\in O$ and $x_0\in V$ such that $a(g_0,x_0)\nin Y$. Let $W\subseteq G$ be an open nbhd of $1_G$ such that $W=W^{-1}$ and $a(W\times\{x_0\})\subseteq V$. We claim that $g_0W\subseteq O$. Let  $h\in W$. Then $h^{-1}\in W$, thus $a(h^{-1},x_0)\in V$ and
$
a(g_0h, a(h^{-1},x_0)) =a(g_0h h^{-1},x_0)=a(g_0,x_0)\nin Y
$
that is, $a(h^{-1},x_0)\nin\dom (\tilde{a}(g_0h))$, i.e. $g_0h\in O$. 
\endproof

\subsection{On the collection of idempotents of $\Gamma(X)$}

We end this section presenting some results about $E(\Gamma(X))$, the idempotents of $\Gamma(X)$, which will be used later to study the Birget-Rhodes expansion of a topological group.   It is easily seen that  $E(\Gamma(X))=\{{\rm id}_V:\; V\su X\; \text{is open}\}$, where ${\rm id}_V$ denotes the identity function on $V$. 

Recall that $CL(X)$ is an inverse semigroup under the operation $\cup$. 

\begin{prop}
\label{hco-compact}
Let $X$ be a locally compact Hausdorff space.  Then the spaces $CL(X$) and $(E(\Gamma(X)), \tau_{hco})$ are topologically isomorphic inverse semigroups. In particular, $E(\Gamma(X))$ is compact with respect to $\tau_{hco}$. 
\end{prop}

\proof Let $\varphi: CL(X)\to E(\Gamma(X))$ given by $\varphi(K)={\rm id}_{K^c}$. Notice that $\varphi(K\cup L) = {\rm id}_{(K\cup L)^c}={\rm id}_{K^c\cap L^c}={\rm id}_{K^c}\circ {\rm id}_{L^c} = \varphi(K)\circ\varphi(L)$. 
It is straightforward to check that  $\varphi$ is continuous.
\endproof

The following observation is straightforward.
\begin{prop}
Let $S$ be a topological inverse semigroup and $T\su S$ be a  subsemigroup (resp. inverse). Then  $\overline{T}$ is a subsemigroup (resp. inverse). 
\end{prop}

\begin{prop}
\label{ext-to-hyp}
Let $X$ be a locally compact Hausdorff  space, $S\su\Gamma(X)$ be an inverse subsemigroup and $Z$ be a topological space. Let $\eta: Z\to (E(S),\tau_{hco})$ be a continuous function.  For  $A\su Z$, we set
\[
\psi(A)=\bigcap_{z\in A}\dom (\eta(z)). 
\]
Then
\begin{itemize}
\item[(i)] If $A$ is compact, then  $\psi(A)$ is open and ${\rm id}_{\psi(A)}\in \overline{S}$, where $\overline{S}$ is the $\tau_{hco}$-closure of $S$.

\item[(ii)] Let $\mathcal{I}:(K(Z),\tau_V)\to (E(\overline{S}),\tau_{hco})$ be defined by $\mathcal{I}(A)={\rm id}_{\psi(A)}$, where $\tau_V$ is the Vietoris topology. Then $\mathcal{I}$ is a continuous extension of  $\eta$. 
\end{itemize}
\end{prop}

\proof (i) Suppose $A\subseteq Z$ is compact. We show that $\psi(A)$ is open.   By Theorem \ref{minimality2}, the following set is open 
\[
Z\ast X= \{(z,x)\in Z\times X:\; x\in \dom(\eta(z))\}.
\]
Let  $x\in \psi(A)$. Since $(z,x)\in Z\ast X$ for all $z\in A$ and $Z\ast X$ is open, there are open sets $V_z\su Z$ and $W_z\su X$ such that $(z,x)\in V_z\times W_z\su Z\ast X$. Since $A$ is compact, there are $z_1, \cdots, z_n\in A$ such that $A\su V_{z_1}\cup \cdots\cup V_{x_n}=V$. Thus  $\psi(V)\su \psi(A)$. 
 Notice that $x\in W_{z_1}\cap \cdots\cap W_{x_n}\su \psi(V)$, thus $\psi(A)$ is open. To see that ${\rm id}_{\psi(A)}\in \overline{S}$, we observe that 
 $\{{\rm id}_{\psi(F)}:\; F\su A\; \text{finite}\}$ is a net (under the partial order $\su$) in $S$  $\tau_{hco}$-converging to ${\rm id}_{\psi(A)}$.
 
(ii) Now we show that  $\mathcal{I}$ is continuous.  Suppose $\mathcal{I}(A)\in \langle K; V\rangle$, where $K\su V\su X$ with $K$ compact and $V$ open.  For each $z\in A$, $\eta(z)\in \langle K;V\rangle$, thus, by the continuity of  $\eta$, there is $V_z\su Z$ open such that $z\in V_z$ and $\eta(V_z)\in \langle K, V\rangle$. Let $U=\bigcup\{V_z:\; z\in A\}$. Then $\mathcal{I}(U^+)\su \langle K; V\rangle$.

Let $V\su X$ be an open subset of $X$. We show that $\mathcal{I}^{-1}({\sf D}^{-1}(V^-))$ is open in $K(Z)$. Take $A\in \mathcal{I}^{-1}({\sf D}^{-1}(V^-))$. Then there is $v\in V$ and  $z\in A$ such that $v\notin \dom(\eta(z))$, thus  $\eta(z)\in {\sf D}^{-1}(V^-)$. By the continuity of  $\eta$, there exists $W\subseteq Z$ open such that $z\in W$ and $\eta(W)\subseteq {\sf D}^{-1}(V^-)$. It is clear that $A\in W^{-}$ and  $\mathcal{I}(W^-)\subseteq {\sf D}^{-1}(V^-)$. In fact, if $B\in W^-$, let $w\in B\cap W\neq\emptyset$. Since, $\eta(w)\in {\sf D}^{-1}(V^-)$, there exists $v'\in V$ such that  $v'\notin \dom(\eta(w))$. Thus, $v'\notin \psi(B)$ and  $\mathcal{I}(B)={\rm id}_{\psi(B)}\in {\sf D}^{-1}(V^-)$.
\endproof

\section{ Birget-Rhodes expansion and Exel's co\-rres\-pondence.}
\label{BR}

%R. Exel \cite{E1} constructed  the monoid $\Sh(G)$ which plays a crucial role in the theory of partial actions and partial representations of groups  and $C^*$-algebras. He  found  a  one-to-one
 %correspondence between partial actions of a group $G$ on a set $X$ and actions of  $\Sh(G)$ on $I(X)$.  On the other hand,  Kellendonk and Lawson  showed that $\Sh(G)$ is isomorphic to the Birget-Rhodes expansion  $\widetilde{G}^R$ of $G$ (see Lemma 2.3 and Theorem 2.4 in \cite{KL}).  
 
We start by recalling  the definition of the  Birget-Rhodes expansion and Exel's correspondence.    The semigroup $\widetilde{G}^R$ was introduced and studied in \cite{BR}, but  we follow the presentation given in  \cite[Proposition 1]{SZ}. Let $\finitosg$ denote the collection of all finite subsets of $G$ and
 \[
\widetilde{G}^R=\{(A,g):\; A\in \finitosg,  \; \{1,g\}\subseteq A\}
\] 
with operation 
\[
(A,g)(B,h)=(A\cup gB, gh).
\]
The set $\widetilde{G}^R$  is an inverse semigroup where  $(A,g)^{-1}=(g^{-1}A, g^{-1}),$ for all $(A,g)\in \widetilde{G}^R.$ 

For the reader's convenience, we present Exel's correspondence in terms of $\widetilde{G}^R$.  Let $S$ be an  inverse monoid and $\theta: G\to S$ be a unital premorphism. Associated to $S$ and $\theta$,  a  semigroup homomorphism $\widetilde\theta:\widetilde{G}^R\to S$ is defined as follows (see  \cite[Theorem 2.4]{KL}): 
\begin{equation}
\label{what}
\widetilde\theta(\{1,g_1,\cdots,g_n\},g_n)=\theta(g_1)\theta(g_1)^{-1}\cdots\theta(g_{n-1}) \theta(g_{n-1})^{-1} \theta(g_n).
\end{equation}
Since $\theta(g) =\widetilde\theta(\{1,g\},g)$,  the map  $\widetilde\theta$ can be regarded as an extension of $\theta$.

Reciprocally, to each semigroup homomorphism $\rho:\widetilde{G}^R\to S$ is associated a unital premorphism $\widehat\rho:G\to S$ as follows: 
\begin{equation}\label{asoc}
\widehat\rho\,(g)=\rho(\{1,g\}, g).
\end{equation}
Kellendonk and Lawson \cite{KL} showed that  $\widehat\rho$ is the unique extension of $\rho$.

The purpose of this section is twofold. On the one hand, we show that the above correspondence can be extended to some inverse semigroups $T\supseteq \widetilde{G}^R$ and $S$.  On the other hand, we also show a topological version of this correspondence where  the group $G$, the semigroup $S$ are topological and the  maps $\theta$ and $\rho$ are continuous. 

\subsection{A generalization of Exel's correspondence}

Let $G$ be a group and $\mathcal{P}^*(G)$ be  the semilattice consisting of non-empty subsets of $G,$ with product $A \bullet B=A\cup B.$ The action $\beta: G\times \mathcal P^*(G)\ni (g, A)\mapsto gA\in \mathcal{P}^*(G)$   allows us to define the semidirect product  $\mathcal{P}^*(G)\rtimes_\beta G$  which is the set $\mathcal{P}^*(G)\times G$ with the operation defined by 
\[
(A,g)(B,h):=(A\bullet \beta_g(B), gh)=(A\cup gB, gh).
\]
then $\mathcal{P}^*(G) \rtimes_\beta G$ is an inverse semigroup. Notice that the  Birget-Rhodes expansion  $\widetilde{G}^R$ is a  subsemigroup of $\mathcal{P}^*(G)\rtimes_\beta G$.
 Since the idempotents in  $\mathcal{P}^*(G)\ltimes_\beta G$  are of the form $(C,1)$ with $C\in \mathcal{P}^*(G), $  the natural partial order is given by
\begin{equation}
\label{order}
(A,g)\leq (B,h)\Longleftrightarrow g=h \,\,\, {\rm and}\,\,\, B\subseteq A.
\end{equation}

It is natural to wonder whether  there is an analogous of  the Exel's correspondence  for  monoids  $T$   with  $\widetilde{G}_R \subseteq T \subseteq \mathcal{P}^*(G) \ltimes_\beta G.$   Such a $T$ will be called an {\it intermediate extension of $\widetilde{G}_R.$}  
Our next purpose is to show  that   the correspondence depends on the inverse semigroup $S$ involved.  

First we introduce a useful notion. Let $T$ be an intermediate extension of $\widetilde{G}_R.$ The {\em support} of $T$, denoted by  ${\bf Supp}(T)$, is defined as follows.
\[
{\bf Supp}(T)=\{A\in \mathcal{P}^*(G):\; (A,g)\in T\; \text{for some $g\in G$}\}. 
\]

Next we show some properties of the support that we need later. 

\begin{lema}
\label{latice}
Let $T$ be an intermediate extension of $\widetilde{G}_R.$ The following assertions hold.
\begin{enumerate}
\item [(i)] Any element of ${\bf Supp}(T)$ contains the identity of $G.$
\item [(ii)] If $T$ is inverse, then ${\bf Supp}(T)$ is a  subsemilattice of  $\mathcal{P}^*(G).$
\end{enumerate}
\end{lema}

\begin{proof}
(i) Let $A$ in ${\bf Supp}(T)$ and $g\in G$ be such that $(A,g)\in T.$ Then $1\in A$, because $(\{1\},1)\in T$ and $(A\cup \{1\}, g )=(\{1\},1)(A,g)=(A,g).$ 

(ii) Suppose that $T$ is an inverse semigroup and let $A,B\in {\bf Supp}(T).$   Let $g\in G$  be such that $(A,g)\in T$. Since $T$ is  an inverse semigroup, the idempotent $(A,1)=(A,g)(g^{-1}A, g^{-1})$ is in $T,$ in the same way we also have that $(B,1)\in T$. Thus $(A\cup B,1)=(A,1)(B,1)\in T$ and $A\cup B\in  {\bf Supp}(T).$
\end{proof}

The following result shows that an extension of Exel's correspondence  fails when $S=\widetilde{G}^R.$  

\begin{prop}
\label{nohom}
Let $G$ be an infinite  group and $T$ be a proper intermediate extension of $\widetilde{G}_R.$ 
There is no semigroup homomorphism   $\mathfrak{i} : T\to \widetilde{G}_R $ extending the  unital premorphism  
\begin{equation}
\label{iota}
 \iota: G\ni g\mapsto (\{1,g\}, g) \in \widetilde{G}_R.
 \end{equation}
\end{prop}

\begin{proof} Suppose not and let  $\mathfrak{i} : T\to \widetilde{G}_R $  be a  semigroup homomorphism with  $\mathfrak{i}(\{1,g\}, g)=(\{1,g\}, g), g\in G$. Since $\widetilde{G}_R \subsetneq  T$, there is an infinite set $A\subseteq G$ such that $(A,g)\in T$ for some $g\in G$. Then   $1\in A,$  thanks to  part (i) of Lemma \ref{latice}.  Write  $\mathfrak{i} (A,g)=(H,h)\in \widetilde{G}_R$  and take  $a\in A.$ Since $1\in A$, we have $(\{1,a\},1)(A,g)=(A,g)$  and  $(H,h)=\mathfrak{i}(\{1,a\},1)\mathfrak{i}(A,g)=(\{1,a\},1)(H,h)$, which implies $a\in H$. Thus $A$ is contained in the finite set  $H$, which is  a contradiction.
\end{proof}

 In view of  Proposition \ref{nohom}, one must impose some restriction on the semigroup $S$  in order to have that any  unital premorphism  $G\to S$ can be extended to a morphism $T\to S$ for any  intermediate extension $T$ of $\widetilde{G}_R.$ With this in mind, we introduce  the next.

\begin{defi}
\label{meetc}
{\rm Let   $\theta: G\to S$ be a unital premorphism and $\mathcal{L}$ be a subsemilattice of  $\mathcal{P}^*(G)$ we say that  $E(S)$ is $(\mathcal{L}, \theta)$-meet complete,  if for any  $A\in \mathcal{L}$  the family of idempotents $\{\theta(a)\theta(a)^{-1}: a\in A\}\subseteq S$  has a meet  (i.e. a greatest lower bound) with respect to the natural partial order.}
\end{defi}

\begin{ex}
\label{ex-meetcomplete}
Let $G$ be a group and $X$ be  a topological space. Then
\begin{itemize}
\item  The semilattice $E( \widetilde{G}_R)$ is   $(\finitosg, \iota)$-meet complete but not    $(\mathcal{P}^*(G), \iota)$-meet complete provided that $G$ is infinite, where $\iota$ is given by \eqref{iota}. %\textcolor{blue}{En general, si $\theta:G\rightarrow S$ es un premorfismo unitario y $\mathcal{L}$ es subsemirretículo de $\finitosg$ \footnote{\color{red} algun ejemplo interesante de eso?}, entonces $E(S)$ es $(\mathcal{L},\theta)$-meet complete.}

\item  Let $\theta: G\to I(X)$ be a partial action and $\mathcal{L}$ be a subsemilattice of  $\mathcal{P}^*(G),$ then  $E(I(X))$ is $(\mathcal{L}, \theta)$-meet complete, because it is   meet complete.

\item Let $G$ be a topological group, $X$ be a locally compact space,   $S\su\Gamma(X)$ be a $\tau_{hco}$-closed inverse subsemigroup  and  $\theta: G\to S$ be  a continuous partial action. It follows from Proposition \ref{ext-to-hyp} that $E(S)$  is  $(K(G), \theta)$-meet complete.
\end{itemize}
\end{ex}
Clearly if  $E(S)$ is $(\mathcal{L}, \theta)$-meet complete, then it is $(\mathcal{L'}, \theta)$-meet complete, for any $\mathcal{L'}$ subsemilattice of $\mathcal{L}.$ Moreover, if either $S$   or $G$ is finite, then $E(S)$ is $(\mathcal{L}, \theta)$-meet complete. 

 Observe that the meet mentioned in the Definition \ref{meetc} is, in principle,  an element of the semigroup $S$. However, as $E(S)$ is an ideal (respect to the natural order), then that meet is actually an idempotent. This justify the definition of the function $\mathcal{I}$ in the following lemma,  which is crucial for the main result of this section.

\begin{lema}
\label{lemainf}
Let $G$ be a group, $S$ be an inverse semigroup, $\theta: G\to S$ be a unital premorphism and $\mathcal{L}$ be a subsemilattice of  $\mathcal{P}^*(G)$ such that  $E(S)$ is $(\mathcal{L}, \theta)$-meet complete. Let $\mathcal{I}: \mathcal{L}\to E(S)$ be defined by
\begin{equation*}
 \label{infi} 
 \mathcal{I}_A = \bigwedge\limits_{a\in A}\{\theta(a)\theta(a)^{-1}: a\in A\},
\end{equation*}
for $A\in \mathcal{L}$. The following assertions hold
\begin{itemize}
\item [(i)] The map  $\mathcal{I}$ is a semigroup homomorphism, in particular  $\mathcal{I} _A  \mathcal{I} _{\{g\}}=\mathcal{I} _A$ for any $A\in \mathcal{L}$ and $g\in A.$ 

\item [(ii)]  Let  $\{A_i\}_{i\in I}$ be a family  in $\mathcal{L}$ such that $\bigcup\limits_{i\in I}A_i  \in \mathcal{L},$ then   the meet $\bigwedge\limits_{i\in I}\mathcal{I}_{A_i}$ exists and equals   $\mathcal{I}_{ \bigcup\limits_{i\in I}A_i }$.

\item [(iii)] $\theta(g)\mathcal{I}_B=\mathcal{I}_{gB}\theta(g)$, for all $g\in G$ and $B\in \mathcal{L}.$

\end{itemize}
\end{lema}

\begin{proof}
As we already observed,  the map $\mathcal{I}$ is well defined. Items $(i)$ and $(ii)$ are straightforward. To prove (iii), take $g\in G$ and $B\in\mathcal{L}.$ Since $\theta$ is a premorphism, $\theta(g)\theta(b)\theta(b^{-1})=\theta(gb)\theta(gb)^{-1}\theta(g)$. Therefore
\begin{align*}
\theta(g)\mathcal{I}_B&=\theta(g)\left(\bigwedge\limits_{b\in B}\theta(b)\theta(b)^{-1}\right)=\bigwedge\limits_{b\in B} \theta(g)\theta(b)\theta(b^{-1})=\left(\bigwedge\limits_{b\in B}\theta(gb)\theta(gb)^{-1}\right) \theta(g)=\mathcal{I}_{gB}\theta(g).  
\end{align*}
\end{proof}

Recall that  a homomorphism $\phi:T\to S $ between inverse semigroups is said to be  {\em meet-preserving}
if for every subset $A \subseteq S$ such that $\bigwedge A$ exists in $T$, then $\bigwedge \phi (A)$ exists in $S$ and
$ \phi (\bigwedge A)=\bigwedge \phi (A).$

Now we present an extension of Exel's correspondence.

\begin{teo}
\label{Text}
Let $\theta:G\to S$ be a unital premorphism between a group $G$ and an  inverse semigroup  $S.$ Consider $T$  an intermediate extension of  $ \widetilde{G}_R$  and  suppose that $T$ is inverse. Then the following statements are equivalent.

\begin{itemize}
\item[(i)]  The semilattice $E(S)$ is $({\bf Supp}(T), \theta)$-meet complete. 
\item [(ii)] $\theta$ has a meet-preserving extension $\theta^*:T \to S.$ 
\end{itemize}
Moreover, suppose that any of the above statements holds. 
Then $\kappa\leq \theta^*,$  for any $\kappa:T\to  S$    homomorphism  extending $\theta,$ and $\kappa=\theta^*$ if and only if $\kappa$ is meet-preserving.

\end{teo}

\begin{proof}
 $(i) \Rightarrow (ii)$ Suppose $E(S)$ is $({\bf Supp}(T), \theta)$-meet complete. Let
\begin{equation}\label{test}\theta^*: T\ni (A,g)\mapsto \mathcal{I}_A\theta(g)\in S.\end{equation}  We have $\theta^*(\{1\},1)=\theta(1)=1$ and $\theta^*(\{1,g\},g)=\theta(g),$ that is $\theta^*$ extends $\theta.$ To check that $\theta^*$ is a homomorphism take  $(A,g)$ and $(B,h)$ in $T.$ As observed in the proof of Proposition \ref{nohom}, one has that $1\in B$ which gives $g\in gB.$ Therefore by items (i) and (iii) of Lemma \ref{lemainf} we obtain
\begin{align*}
		\theta^*(A,g)\theta^*(B,h)&=
\mathcal{I}_A\mathcal{I}_{gB}\theta(g)\theta(h)=\mathcal{I}_A\mathcal{I}_{gB}\theta(g)\theta(g)^{-1}\theta(gh)=\mathcal{I}_A\mathcal{I}_{gB}\theta(gh)=\mathcal{I}_{A\cup gB}\theta(gh),
	\end{align*}
and $\theta^*(A,g)\theta^*(B,h)=\theta^*((A,g)(B,h)).$ To show that $\theta^*$ is meet-preserving consider a family $\{(A_i, g_i)\}_{i\in I}\subseteq T$ such that $\bigwedge\limits_{i\in I}(A_i,g_i)$ exists in  $T.$ By \eqref{order}, there exists $g\in G$ such that $g_i=g$, for all $i\in I$ and $\bigwedge\limits_{i\in I}(A_i,g_i)=\left(\bigcup\limits_{i\in I}A_i,g\right).$  Thus  $\{A_i\}_{i\in I}\subseteq {\bf Supp}(T)$ satisfies the assumptions of $(ii)$ in  Lemma  \ref{lemainf}, then 
$$
\theta^*\left(\bigwedge\limits_{i\in I}(A_i,g_i)\right)=\mathcal{I}_{\bigcup\limits_{i\in I}A_i}\theta(g)=\bigwedge\limits_{i\in I}\mathcal{I}_{A_i}\theta(g)=\bigwedge\limits_{i\in I}\theta^*(A_i, g).
$$

\medskip

\noindent $(ii) \Rightarrow (i) $  Let  $\theta^*: T\to S$ be a meet-preserving extension of $\theta$ and take $A\in {\bf Supp}(T),$ by $(i)$ in Lemma \ref{latice} we know that $(A,1)\in T$. Since $(A,1)=\bigwedge\limits_{a\in A}\{(\{1,a\}, 1)\}$ and $(\{1,a\}, 1)=(\{1,a\}, a)(\{1,a^{-1}\}, a^{-1}),$  we  have
\begin{equation}
\theta^*(A,1)=\bigwedge\limits_{a\in A}\{\theta^*(\{1,a\}, 1)\}=\bigwedge\limits_{a\in A}\{\theta(a)\theta(a)^{-1}\}=\mathcal{I}_A.
\end{equation}
Thus $E(S)$ is $({\bf Supp}(T), \theta)$-meet complete.

	Let  $\kappa:T\to  S$ be a homomorphism extending $\theta$ and let $\theta^*$ be defined by \eqref{test}. Take $(A,g)\in T,$  we know that $(A,1)\in T$,  since  $T$ is  inverse.  On the other hand, for $a\in A$  the equality $(\{1,a\},a)(\{1,a^{-1}\},a^{-1})(A,1)=(A,1)$ gives $\kappa(A,1)=\theta(a)\theta(a)^{-1}\kappa(A,1)$ therefore $\kappa(A,1)\leq \theta(a)\theta(a)^{-1}$ for all $a\in A$. Hence, $\kappa(A,1)\leq \mathcal{I}_A$. Since the natural partial order is compatible with  product,  one gets
\begin{center}
$\kappa(A,g)=\kappa(A,1)\kappa(\{1,g\},g)\leq\mathcal{I}_A\kappa(\iota(g))=\mathcal{I}_A\theta(g)=\theta^*(A,g),$
\end{center}
and $\kappa\leq \theta^*.$ To check the last assertion, suppose that $\kappa$ is a meet-preserving homomorphism, again the equality  $(A,1)=\bigwedge\limits_{a\in A}\{(\{1,a\}, 1)\}$ gives  $\kappa(A,1)=\mathcal{I}_A$
which implies $\kappa(A,g)=\theta^*(A,g),$ as desired.
\end{proof}

\begin{rem}
{\rm Taking $T=\widetilde{G}_R$ in  Theorem \ref{Text} and using the first item of Example \ref{ex-meetcomplete} we recover Exel's correspondence presented at the beginning of the section.}
\end{rem}

We give another application of Theorem \ref{Text}.

\begin{ex}\label{exten}
 {\rm Let $G$ be a topological group  and   $\lambda_g$ be the left multiplication by $g\in G.$ It is not difficult to show that 
 \begin{equation}
 \label{lambdag}\mathfrak{l}: G\ni g\rightarrow \lambda_g|_{\{1,g^{-1}\}^c}\in  \Gamma(G\setminus\{1\})
 \end{equation}

is a unital premorphism. For a subset  $A$ of $G$  with $1\in A$ we have  $\bigcap\limits_{g\in A}\{1,g\}^{c}=A^c$ and $\mathcal{I}_A={\rm id}_{A^c}$. Thus $E(\Gamma(G\setminus\{1\}))$ is  $(CL_1(G),\Lambda)$-meet complete,  where $CL_1(G)$ is the semilattice of closed subsets of $G$ containing $1$. Let $T=\{(A,g):\; A\in CL(G),\, \{1,g\}\subseteq A\},$   by Theorem \ref{Text} the map $\mathfrak{l}^*: T\ni (A,g)\mapsto {\rm id}_{A^c}\circ\mathfrak{l}(g)\in  \Gamma(G\setminus\{1\})$  is the only  meet-preserving monoid homomorphism extending $\mathfrak{l}.$ }
\end{ex}

\subsection{A topological correspondence}

In this subsection we present  Exel's correspondence from a topological point of view, that is to say, we assume that $G$ and $S$ are topological spaces and $\theta$ is continuous. 

When $G$ is  a topological group, a natural way of endowing  $\widetilde{G}^R$ with a topology is seeing  $\finitosg$ as a subspace of the hyperspace $K(G)$ with the Vietoris topology.   This was  analyzed by K. Choi  in \cite{Choi2013}, he defined an intermediate extension of  $\widetilde{G}^R$ as follows:
\[
\widetilde{G}_c^R=\{(A,g)\in K(G)\times G:\; \;\{1,g\}\subseteq A\}.
\]
The following is a consequence of \cite[Proposition 2.6]{Choi2013}. 
 
\begin{prop}
\label{top-choi}
Let $G$ be a topological group. Then  $\widetilde{G}_c^R$ is a topological inverse monoid with the topology inherited from $K(G)\times G$.
\end{prop}

The topological version of Exel's correspondence we are about to present is given  for topological inverse semigroups whose set of idempotents is a lattice of special type.  A {\it topological semilattice}  is a Hausdorff topological space $E$ endowed with a partial order  such that any two elements have a greatest lower bound and the map $E\times E\ni (x,y)\mapsto x\wedge y\in E$ is continuous. A topological semilattice $E$ has {\it small semilattices}, if   every point has a basis of neighbourhoods  which are subsemilattices of $E$ (see \cite{CCL,Law1}).

%We say that a base $\mathcal{B}$  for a topological inverse semigroup $S$  is {\em idempotent-closed} if $E(S)\cap V$ is closed under  product for all $V\in \mathcal{B}$. 
\medskip

\begin{ex}
\label{ex-non}
{\rm It is clear that if $S$ is a topological inverse semigroup then $E(S)$ with the natural partial order is a topological semilattice. Moreover,  for each of the following topological inverse semigroups, its set of idempotents admits small semilattices. %semigroups admit idempotent-closed basis.

\begin{itemize}
\item[(i)]   $(\Gamma(X), \tau_{hco})$ where  $X$ is a locally compact Hausdorff space.  

\item[(ii)] $\widetilde{G}^R$  and $\widetilde{G}_c^R$ for every topological group  $G$. 

\end{itemize}}

\end{ex}

\begin{rem}{\rm %Let $S$ be a topological  Hausdorff inverse semigroup, then $E(S)$ is a topological semilattice. Thus, if $S$ is idempotent-closed, $E(S)$ has small semilattices, (see \cite[p. 209]{Law}). 
Topological semilattices having small semilattices are also known  as Lawson semilattices. For a concrete example of a non Lawson semilattice, the interested reader may consult  \cite[Example 2]{La} and \cite[Theorem 4.5, pag. 296]{CCL}. }
\end{rem}

%In the next result we show how to obtain semigroups with idempotent-closed basis from Lawson semilattices.
%\begin{prop}
%	Let $S$ be a locally compact semilattice, $G$ be a topological group and $\alpha: G\times S\rightarrow S$ be an action of $G$ on $S$ by  endomorphisms. Then $S$ has small-semilattices if and only if  the semidirect product $S\ltimes_{\alpha}G$ has a idempotent-closed basis.
%\end{prop}
%
%
%
%\begin{proof}
%	Suppose that  $S$ tiene small-semilattices, and  let $\mathcal{U}$ be a  base of  subsemilattices of $S$. Further, take  $\mathcal{V}$ a basis  $G$
% then  $\mathcal{B}=\{U\times V\subseteq S\ltimes_{\alpha}G: U\in\mathcal{U}\ y\ V\in\mathcal{V}\}$ is a basis of  $S\ltimes_{\alpha}G$. Moreover, for  $U\times V\in \mathcal{B}$ and  $(s,1), (t,1)\in E(S\ltimes_{\alpha}G)\cap  (U\times V)$  we see that $(s,1)(t,1)=(st,1)\in U\times V$ because $st\in U$, this shows that $\mathcal{B}$ is an idempotent-closed basis for $S\ltimes_{\alpha}G$. Conversely suppose that  $S\times_{\alpha}G$  has an idempotent-closed basis, then  $E(S\ltimes_{\alpha}G)$ has  small-semilattices. But  $E(S\ltimes_{\alpha}G)=S\times\{1\}$ then  $S$ has small-semilattices.
%\end{proof}

The following lemma is crucial for the proof of the main result of this section. 
\begin{lema} 
\label{smsl}
If $E$ is a topological semilattice with small-semilattices, then 
\begin{equation}
\label{pic}
\pi: \mathcal P_{\rm fin}(E) \ni  \{x_1,\cdots,x_n\}\mapsto x_1\cdots x_n\in E
\end{equation}
is continuous, where $\mathcal P_{\rm fin}(E)$ has the Vietoris topology.
\end{lema}

\begin{proof}
Let  $\mathcal{B}$ be a basis of subsemilattices for $E$. Let $V\in\mathcal{B}$ and  $A=\{x_1,\cdots, x_n\}\in \mathcal{P}_{\rm fin}(E)$ be such that $\pi(A)\in V$.  
	Let  $\alpha:E^n\to E$ given by $\alpha(f_1,\cdots, f_n)=f_1\cdots f_n$. As $\alpha$ is continuous, there are open sets $U_1, \cdots, U_n$ in $X$ such that  $x_i\in U_i$ and $\alpha(U_1\times\cdots\times U_n)\su V$.
	Consider the following open set in $\mathcal{P}_{\rm fin}(E)$
	\[
	M=N(U_1, \cdots, U_n)\cap \mathcal{P}_{\rm fin}(E).
	\]
	Clearly $A\in M$ and we claim that $\varphi(M)\su V$. Indeed, let $B=\{y_1, \cdots, y_m\}\in M$.  If $m\leq n$, then $\varphi(B)=\alpha(y_1', \cdots, y_n')$ where $y_i'\in U_i$ and $B=\{y_1', \cdots, y_n'\}$ (repeating some elements if necessary, as $B\cap U_i\neq \emptyset$ for all $i$). Since $\alpha(y'_1, \cdots, y'_n)\in V$ there is nothing to show.
	Thus, we can assume that $n<m$ and  $y_i\in U_i$, $i=1,\cdots,n$.  We show the case $m=n+1$ which gives the general pattern. As $B\in M$, there is $k\leq n$ such that $y_{n+1}\in U_k$. Let $a=(y_1, \cdots, y_n)$ and $b=(y_1,\cdots, y_{k-1}, y_{n+1},\cdots, y_{n})$, i.e. we substitute $y_k$ by $y_{n+1}$  in $a$. We have
	\[
	y_1\cdots y_n y_{n+1}\,=\,[y_1\cdots y_n] [y_1\cdots y_{k-1}y_{n+1}\cdots y_{n}]=\alpha(a)\alpha(b).
	\]
	Thus $\varphi(B)=\alpha(a)\alpha(b)$.  Notice that $a,b\in U_1\times\cdots\times U_n$, thus $\alpha(a), \alpha(b)\in V$.
	Since $V$ is closed under multiplication, $\varphi(B)\in V$.
\end{proof}
%
%{\color{red}Creo que podemos eliminar el lema que sigue e incorporar ese argumento sencillo dentro de la prueba de \ref{ExelCorres}}

%\begin{lema} \label{inducida}
%Let $S$  be a topological inverse semigroup. Then the following assertions are equivalent.
%\begin{enumerate}
%		\item [(i)] For any topological space $X$ and any continuous map  $\eta: X\rightarrow E(S)$,  the map $\varphi: \finitosx \rightarrow E(S)$, $\{ x_1,\cdots,x_n\}\mapsto \eta(x_1)\cdots\eta(x_n)$ is continuous.
%		\item [(ii)]The map $\pi: \mathcal{P}_{\rm fin}(E(S))\rightarrow E(S)$, $\{ x_1,\cdots,x_n\}\mapsto x_1\cdots x_n$ is continuous.
%	\end{enumerate}
%\end{lema}

%\begin{proof}
%We only need to show $(ii)\Rightarrow (i)$. Since  $\eta$ is continuous,  then so is the map   but   thus  $\varphi$ is continuous.
%\end{proof}

We recall that $\widetilde{G}^R$ has the topology inherited as a topological inverse submonoid of $\widetilde{G}_c^R$ (see Proposition \ref{top-choi}). 
The topological version of Exel's correspondence is the following. 

\begin{teo} 
\label{ExelCorres}
Let $G$ be a topological group and $S$ be a topological inverse semigroup such that $E(S)$ has small-semilattices. 
\begin{itemize}
\item[(i)] If $\theta: G\to S$ is a continuous unital premorphism, then its associated semigroup  homomorphism $\widetilde\theta:\widetilde{G}^R\to S$ defined by  \eqref{what}  is  continuous. 

\item[(ii)] If $\rho:\widetilde{G}^R\to S$ is a continuous semigroup homomorphism, then its associated unital premorphism  $\widehat\rho:G\to S$ given by \eqref{asoc} is continuous. 
\end{itemize}

\end{teo}

\begin{proof}
(i) Suppose $\theta$ is a continuous unital premorphism  as in the hypothesis. Let $\eta:G\to E(S)$ be given by $\eta(g)=\theta(g)\theta(g)^{-1}$ and $\varphi:\finitosg\to E(S)$ defined  by $\varphi(\{g_1, \cdots, g_n \}) = \eta(g_1)\cdots \eta(g_n)$. Then $\widetilde\theta(A,g)=\varphi(A)\theta(g)$. It is enough to show that $\varphi$ is continuous.  Notice that  $\varphi=\pi\circ I_{\eta},$  where  $\pi:\mathcal P_{\rm fin}(E(S)) \to E(S) $  is the map  defined in \eqref{pic}  and  $I_{\eta}: \mathcal{P}_{\rm fin}(G) \ni \{x_1,x_2,\cdots,x_n\}\mapsto \{\eta(x_1),\cdots,\eta(x_n)\}\in \mathcal{P}_{\rm fin}(E(S))$. By  Lemma \ref{smsl} ,  $\pi$ is continuous and clearly $I_\eta$ is also continuous, therefore so is $\varphi$, as desired.
	
	(ii) This is straightforward, as the map $g\mapsto \{1,g\}$ from $G$ to $\finitosg$ is continuous.
\end{proof}

\begin{rem}
{\rm By  \cite[Theorem 2.4]{KL},  $\widetilde\theta$ is the only continuous homomorphism from $\widetilde{G}^R$ to $S$  extending $\theta.$}
\end{rem}

\medskip

A {\em topological action of an inverse semigroup} $S$ on a topological space $X$ is a  homomorphism from $S$ to $\Gamma(X).$ 
By Theorem   \ref{minimality3}, we say that an action of $S$ on $X$ is {\em continuous} when it is a continuous homomorphism from $S$ to $(\Gamma(X), \tau_{hco})$. Thus, by Theorem \ref{ExelCorres}, we have the following .

\begin{coro}
\label{topol}
Let $X$ be a locally compact Hausdorff space. There is a one-to-one correspondence between continuous  partial action of $G$ on $X$ and continuous actions of $\widetilde{G}^R$ on $X.$
\end{coro}

The category of  continuous partial actions  of a topological group $G$  on locally compact Hausdorff spaces together with continuous morphisms forms a subcategory of the category of partial actions of  $G$ and morphisms (see \cite[p. 17]{abadie}).  The proof of the following is now immediate from Corollary \ref{topol} and a standard 
category theory.

\begin{prop}
The function $\theta \mapsto \tilde\theta$ is the object part of a functor from the
category of  continuous partial actions  of a topological group $G$  on locally compact Hausdorff spaces to the category
of continuous actions of $\widetilde{G}^R$  on locally compact Hausdorff spaces and this functor determines an isomorphism between these categories.
\end{prop}

We show next a result  analogous to Theorem \ref{ExelCorres} but now for  $T=\widetilde{G}_c^R$ and $S\su\Gamma(X)$. 

\begin{teo} 
\label{ChoisExtension}
Let $G$ be a topological group and $X$ a locally compact Hausdorff space and $S\su \Gamma(X)$ be a $\tau_{hco}$-closed inverse subsemigroup. Suppose  $\theta: G\to (S, \tau_{hco})$ is a continuous unital premorphism. Then $\theta$ can be extended to a continuous semigroup  homomorphism $\theta^*:\widetilde{G}_c^R\to (S,\tau_{hco})$. 
\end{teo}

\begin{proof}
 First we verify that we can applied Theorem \ref{Text}. Clearly, ${\bf Supp}(\widetilde{G}_c^R)=K(G)\setminus\{\emptyset\}$ and thus  $E(S)$  is  $({\bf Supp}(\widetilde{G}_c^R), \theta)$-meet complete (see Example \ref{ex-meetcomplete}). Hence $\theta^*(A,g)=\mathcal{I}_A\theta(g)$ is an extension as desired. We only need to check that it is continuous. Indeed,  it follows from 
Lemma \ref{ext-to-hyp} that $\mathcal{I}:K(G)\to (E(S),\tau_{hco})$ is continuous. 
\end{proof}

\begin{ex}
{\rm  Let $X$ be locally compact Hausdorff  space  and $G$ be a topological group. Let  $a:G\times X\to X$ be a continuous action and  $Y\subseteq X$ be an open set. By Proposition \ref{inducon}, the induced partial action  $\theta:G\to  (\Gamma(Y), \tau_{hco})$ is continuous, thus by  Theorem \ref{ChoisExtension} the map  $\theta$ has a  continuous  extension $\theta^*: \widetilde{G}_c^R\to \Gamma(Y).$}
\end{ex}

Our next goal is to show that  the map $\theta^*$ in Theorem \ref{ChoisExtension}  is not unique.  Let $\lambda_g$ the left multiplication by $g\in G$ it  follows  from \cite[(2) Proposition 2.2]{Choi2013}  the map 
$$
\tilde\Lambda: \tilde{G}^R_c\ni (A,g)\mapsto \lambda_g|_{\lambda_{g^{-1}}(A^c)} \in \Gamma(G\setminus\{1\})
$$
%\begin{equation*} \label{La}
%\tilde\Lambda(A,g)=\begin{cases}
%	\lambda_g|_{\lambda_{g^{-1}}(A^c)},\ \text{if}\ A\neq \{1\}\\
%	{\rm id}_G,\ \text{if}\ A=\{1\}
%\end{cases}
%\end{equation*} 
%\begin{equation}\label{La}
%	\Lambda: \tilde{G}^R_c \ni (A,g)\mapsto \lambda_g|_{\lambda_{g^{-1}}(A^c)}\in \Gamma(G)
%\end{equation}
is a  monoid homomorphism which turns out to be a monomorphism provided that $G$ is not compact.

 In the following Lemma we study the continuity of $\tilde\Lambda$ when $\Gamma(G)$ has the topology $\tau_{hco}$.
\begin{prop}\label{choicon}
	 Let $G$ be locally compact Hausdorff group. Then $\tilde\Lambda$  is continuous. 
\end{prop}
\begin{proof} Let  $(A,g)\in \tilde{G}^R_c$ be such  $\tilde\Lambda(A,g)\in \langle K,V\rangle$. Then  $K\subseteq g^{-1}\cdot A^c$ and $gK\subseteq V$. Since $K$ is compact,  there is a neighborhood $U$ of  $g$ in $G$ such that $UK\subseteq V$. Moreover,  $g^{-1}A\subseteq K^c$ and by the compactness of  $A$,  there are open sets  $W$ and   $Z$ of $G$ for which  $g^{-1}\in W$, $A\subseteq Z$ and  $WZ\subseteq K^c$. It is easy to see that  $g\in L:=W^{-1}\cap U$ and  $(A,g)\in J:=(N(Z)\times L)\cap \tilde{G}^R_c$.  Moreover,  for  $(B,h)\in J$, we have $h^{-1}B\subseteq WB\subseteq WZ\subseteq K^c$, and $K\subseteq \dom(\tilde\Lambda(B,h))$. Hence   $hK\subseteq UK\subseteq V$ and  $\tilde\Lambda(J)\subseteq \langle K,V\rangle$.
	
 On the other hand, take $V$ a non-empty open set  $G$. We show that  $\tilde\Lambda^{-1}({\sf D}^{-1}(V^-))$ is open. Note that  $(A,g)\in \Lambda^{-1}({\sf D}^{-1}(V^-))$ if and only if   $g^{-1}\cdot A\cap V\neq \emptyset$. In this sense if  $(A,g)\in \tilde\Lambda^{-1}({\sf D}^{-1}(V^-))$, there exists $a\in A$ such that  $g^{-1}\cdot a\in V$, and  open sets $W_0$ y $Z_0$ of  $G$ with  $g^{-1}\in W_0$, $a\in Z_0$ and $W_0\cdot Z_0\subseteq V$. We have  $(A,g)\in Y:=(N(G,Z_0)\times W_0^{-1})\cap \tilde{G}^R_c$. Also, for  $(B,h)\in Y$, there is  $b\in B\cap Z_0$ and $h^{-1}\cdot b\in W_0\cdot Z_0\subseteq V$,  then  $h^{-1}\cdot B\cap V\neq \emptyset$ which gives $(B,h)\in \Lambda^{-1}({\sf D}^{-1}(V^-))$.  This shows  $Y\subseteq \Lambda^{-1}({\sf D}^{-1}(V^-))$ and  the set $\tilde\Lambda^{-1}({\sf D}^{-1}(V^-))$ is open. Finally, as  ${\sf I}^{-1}(V^-)=i^{-1}({\sf D}^{-1}(V^-))$, we have  $\tilde\Lambda^{-1}({\sf I}^{-1}(V^-))$ is open in $\tilde{G}^R_c$.
\end{proof}
\begin{ex}\label{fex}{\rm
 Suppose  $G$ is  a  locally compact Hausdorff  group, by Proposition \ref{choicon}  the map $\Lambda$ defined in \eqref{lambdag}  
is a  continuous premorphism  which is  clearly extended by the  continuous homomorphism $\tilde\Lambda,$ also by Theorem \ref{ChoisExtension} the map $\mathfrak{l}^*$ defined in Example \ref{exten} is another continuous homomorphism extending $\Lambda,$ in general $\tilde\Lambda\neq \mathfrak{l}^*.$ For a concrete example take $G=\mathbb{R}.$  Then  $\tilde\Lambda ([0,1],0)={\rm id}_{[0,1]^c}$ and $\mathfrak{l}^*([0,1],0)={\rm id}_{\{0,1\}^c}.$}
\end{ex}

\bigskip

{\bf Acknowledgements:} We are very thankful to the referee for all their  comments that improved the presentation of the paper.  This work was partially supported by grants No. 4.583 of Fundaci\'on para la Promoci\'on de la Investigaci\'on y la Tecnolog\'ia del Banco de La Rep\'ublica (Colombia) and project VIE-8041 of Universidad Industrial de Santander.

\end{document}